\documentclass{article}
   \usepackage[cp1251]{inputenc}
\usepackage{amsfonts,amsmath,amsthm,amscd,amssymb,amsbsy,mathrsfs}
    \usepackage[english]{babel}

\newtheorem{theorem}{Theorem}[section]

\newtheorem{lemma}[theorem]{Lemma}
\newtheorem{proposition}[theorem]{Proposition}
\theoremstyle{definition}
\newtheorem{definition}[theorem]{Definition}
\theoremstyle{remark}
\newtheorem{remark}[theorem]{Remark}
\numberwithin{equation}{section}

\begin{document}

\title {Two-dimensional moment problem and Schur algorithm}

\author{Ivan Kovalyov and Stefan Kunis}



\maketitle

\begin{abstract}
We study a truncated two-dimensional moment problem in terms of the Stieltjes transform. The set of the  solutions is described by the Schur step-by-step algorithm, which is based on the continued fraction expansion of the solution. In particular, the obtained results are applicable  to the two-dimensional moment problem for  atomic measures.
\end{abstract}

\tableofcontents

\section{Introduction}

A  Stieltjes  moment problem was studied in~\cite{Akh,St89}.
Given a sequence of  real numbers $\textbf{s}=\{s_{i}\}_{i=0}^{\infty}$, find a
positive Borel measure $\sigma$ with a support on $\mathbb{R}_{+}$, such
that
\begin{equation}\label{int1}
    \int_{\mathbb{R}_{+} }t^jd\sigma(t)=s_{j},\qquad j\in\mathbb{Z}_{+}=\mathbb{N}\cup\{0\}.
\end{equation}
The problem \eqref{int1} with a finite  sequence
 $\mathbf{s}=\{s_{i}\}_{i=0}^{\ell}$
is called a truncated Stieltjes   moment problem, otherwise  it is called a full Stieltjes moment problem.

By the Hamburger--Nevanlinna theorem~\cite{Akh}, the truncated Stieltjes   moment problem can be reformulated in term
of the Stieltjes transform
\begin{equation*}\label{int3}
    f(z)=\int_{\mathbb{R}_{+}}\frac{d\sigma(t)}{t-z},\qquad z\in \mathbb{C}\backslash \mathbb{R}_{+},
\end{equation*}
of $\sigma$ as the following interpolation problem at infinity
\begin{equation*}\label{14.int.th_1}
	 f(z)=-\frac{s_{0}}{z}-\frac{s_{1}}{z^2}-\cdots-\frac{s_{\ell}}{z^{\ell+1}}
    +o\left(\frac{1}{z^{\ell+1}}\right),\quad\quad z\widehat{\rightarrow}\infty,
\end{equation*}
where the notation $z\widehat{\rightarrow}\infty$ means that $z\rightarrow\infty$  nontangentially, that is inside the
sector $\varepsilon<\arg z<\pi-\varepsilon$ for some
$\varepsilon>0$. 

Recall the result describing the first step of the Schur algorithm.
\begin{theorem}\label{14p.int.th_ind} (\cite{Akh}) Let $\mathbf{s}=\{s_{j}\}_{j=0}^{\ell}$ be a sequence of real numbers with $s_0\neq 0$ and let $f$ admit the asymptotic expansion \eqref{14.int.th_1}
Then $f$ takes the form
\begin{equation}\label{14.int.th_2}
	f(z)=-\cfrac{b_0}{a_0(z)+f_1(z)},
\end{equation}
where 
\begin{equation}\label{14.int.th_3}
	b_0=s_0, \quad a_0(z)=z-\cfrac{s_1}{s_0},\quad\mbox{and}\quad f_1(z)=-\sum_{j=0}^{\ell-2}\cfrac{s_j^{(1)}}{z^j},
\end{equation}
the recursive sequence $\mathbf{s}^{(1)}=\left\{s_j^{(1)}\right\}_{j=0}^{\ell-2}$ is defined by
 \begin{equation}\label{14.int.th_4}
s_{j}^{(1)}=\cfrac{(-1)^{j+1}}{s_0^{j+3}}\begin{vmatrix}
s_1& s_0 &0&\ldots& 0\\
\vdots& \ddots &\ddots&\ddots & \vdots\\
\vdots&  &\ddots&\ddots &0\\
\vdots&  & &\ddots & s_0\\
s_{j+2}& \ldots &\ldots&\ldots& s_1
\end{vmatrix},\quad j=\overline{1,\ell-2}.
\end{equation}

\end{theorem}

\textbf{Two-dimensional moment problem}. Now let $\mu$ be a nonnegative Borel measure on $\mathbb{R}^2$, where supp$(\mu)=A\subseteq \mathbb{R}^2$. The moment sequence 
 $\mathbf{s}=\{s_{i,j}\}_{i,j=0}^\ell$  is defined by
\begin{equation}\label{14.2.1}
s_{i,j}=\int\int_A t^i\tau^j d\mu(t,\tau).
\end{equation}

However, we can reformulate the moment problem \eqref{14.2.1}  in terms of the generalized moments.
 Let us define a linear functional  $\mathfrak{S}$ on the monomials $t^i\tau^j$ by
\begin{equation}\label{14.2.3S}
	\mathfrak{S}(t^i\tau^j)=s_{i,j}.
\end{equation}

In the present paper,  we study a two--dimensional problem in a similar way as the one--dimensional problem \eqref{14.int.th_1}, i.e., we   investigate the two--dimensional problem as the interpolation problem. Recall  a Stieltjes transform in two variables (see \cite{Cuyt})
\begin{equation}\label{14.int. st tr.1}
\int_0^\infty \int_0^\infty \cfrac{d\mu(t,\tau)}{1+\tilde zt+\tilde\zeta\tau}=\sum_{i,j=0}^\infty {i+j\choose i}(-1)^{i+j}s_{i,j}{\tilde z}^i{\tilde \zeta}^j, \quad \tilde{z},\tilde\zeta\in\mathbb{C}^2
\end{equation}

Hence, it is more convenient to define the associated function $F$ by
\begin{equation}\label{14.2.2}
 F(z, \zeta)=-\frac{1}{z\zeta}\int\int_A\frac{d\mu(t,\tau)}{1-\frac{t}{z}-\frac{\tau}{\zeta}}, \quad \mbox{where}\,\left|\frac{t}{z}+\frac{\tau}{\zeta}\right|<1, \, \tilde z=-\cfrac{1}{z}, \, \tilde\zeta=-\cfrac{1}{\zeta}.
\end{equation}

Moreover, \eqref{14.2.2} can be rewritten as
\begin{equation}\label{14.2.3}\begin{split}
    F(z, \zeta)&= -\frac{1}{z\zeta}\int\int_A\sum_{k=0}^\infty\left(\frac{t}{z}+\frac{\tau}{\zeta}\right)^k d\mu(t,\tau)=\\&=-\sum_{i,j=0}^\infty{i+j\choose i} s_{i,j}\frac{1}{z^{i+1}}\frac{1}{\zeta^{j+1}}=-\sum_{i=0}^{\infty}\frac{1}{z^{i+1}}\sum_{j=0}^\infty {i+j\choose i}\frac{s_{i,j}}{\zeta^{j+1}}=\\&=
    -\frac{1}{z}\sum_{j=0}^{\infty}\frac{s_{0,j}}{\zeta^{j+1}}-\frac{1}{z^2}\sum_{j=0}^{\infty}\frac{(j+1)s_{1,j}}{\zeta^{j+1}}-\frac{1}{z^3}\sum_{j=0}^{\infty}\frac{(j+1)(j+2)s_{2,j}}{\zeta^{j+1}}-\ldots=\\&\!=\!-\frac{1}{z}\sum_{j=0}^{n\!-\!2}\frac{s_{0,j}}{\zeta^{j\!+\!1}}\!-\!\frac{1}{z^2}\sum_{j=0}^{n\!-\!3}\frac{(j\!+\!1)s_{1,j}}{\zeta^{j\!+\!1}}\!-\!\ldots
    -\frac{1}{z^{n\!-\!2}}\sum_{j=0}^{1}{n\!+\!j\!-\!3\choose j}\frac{s_{n-3,j}}{\zeta^{j+1}}-\\&-\frac{s_{n-2,0}}{z^{n-1}\zeta}+o\left(\sum\limits_{k=0}^n\frac{1}{z^k \zeta^{n-k}}\right).
\end{split}
\end{equation}

Now, we briefly describe the content of the paper. The truncated two-dimensional problem of the non-symmetric form is solved in Section 2.  Section 3 contains the truncated two-dimensional problem in symmetric form and the description of its solutions. A so-called Stieltjes like case is studied in the Section 4, where  we find  the solutions of the truncated  two-dimensional problem in terms of  Slieltjes like fractions. A truncated two-dimensional problem for   atomic measures is studied in the Section 5.   The set of solutions of the full  two-dimensional problem is described in Section~6.

\section{Non-symmetric case}
In the current section, we study a truncated two-dimensional moment problem in the non-symmetric form, i.e., the solutions are not symmetric  with respect to the variables $z$ and $\zeta$. 

\begin{definition} Let $\mathbf{s}=\{s_j(z_1,\ldots, z_n)\}_{j=0}^{\ell}$ be a sequence such that $s_j(z_1,\ldots, z_n)$ depend of the variables $z_1,\ldots, z_n$. $\mathfrak{N}(\mathbf{s})$ is called a simple regular set of the sequence $\mathbf{s}$  defined by
\begin{equation}\label{14.2.18*}
	\mathfrak{N}(\mathbf{s})=\{(z_1,\ldots, z_n) |\quad D_i(z_1,\ldots, z_n) \neq 0 \quad\mbox{for all}\quad i=\overline{0,[\ell/2]}  \},
\end{equation}
where
\begin{equation}\label{14.2.18}
	D_i(z_1,\ldots, z_n) =\begin{vmatrix}
s_0(z_1,\ldots, z_n)& \ldots& s_{i-1}(z_1,\ldots, z_n)\\
  \vdots &\ddots &\vdots \\
s_{i-1}(z_1,\ldots, z_n) & \ldots  & s_{2i-2}(z_1,\ldots, z_n)
\end{vmatrix}.
\end{equation}

Moreover, $\mathbf{s}=\{s_j(z_1,\ldots, z_n)\}_{j=0}^{\ell}$ is called a simple regular sequence on the set  $\mathfrak{N}(\mathbf{s})$.
\end{definition}

\textbf{Truncated  problem $MP(\mathbf{s})$:}
Given a sequence of real numbers $\mathbf{s}=\{s_{i,j}\}_{i,j=0}^{i+j= n-2}$, describe the set $\mathcal{M}(\mathbf{s})$ of  functions $F$, which have the following asymptotic expansion
\begin{equation}\label{14.2.4}\begin{split}
	F(z,\zeta)&=-\frac{1}{z}\sum_{j=0}^{n-2}\frac{s_{0,j}}{\zeta^{j+1}}-\frac{1}{z^2}\sum_{j=0}^{n-3}\frac{(j+1)s_{1,j}}{\zeta^{j+1}}-\ldots
    -\\&-\frac{1}{z^{n-2}}\sum_{j=0}^{1}{n+j-3\choose j}\frac{s_{n-3,j}}{\zeta^{j+1}}-\frac{s_{n-2,0}}{z^{n-1}\zeta}+o\left(\sum\limits_{k=0}^n\frac{1}{z^k \zeta^{n-k}}\right).\end{split}
\end{equation}

Setting $s_k(\zeta)=\sum\limits_{j=0}^{n-k-2}{j+k\choose j}\frac{s_{k,j}}{\zeta^{j+1}}$, then \eqref{14.2.4}  can be rewritten as 
\begin{equation}\label{14.2.5}
	F(z,\zeta)=-\frac{s_0(\zeta)}{z}-\frac{s_1(\zeta)}{z^2}-\ldots--\frac{s_{n-2}(\zeta)}{z^{n-1}}+o\left(\frac{1}{z^{n-1}}\right).
\end{equation}

\subsection{Basic truncated    problem}

Let us consider a basic problem.  Given a sequence of real numbers $\mathbf{s}=\{s_{0,0}, s_{0,1}, s_{1,0}\}$, such that 
\begin{equation}\label{14.2.6}
	\frac{s_{0,0}}{\zeta}+\frac{s_{0,1}}{\zeta^2} \neq 0.
\end{equation}
Describe the set 
$\mathcal{M}(\mathbf{s})$ of  functions $F$, which have the  asymptotic expansion
\begin{equation}\label{14.2.7}
	F(z,\zeta)=
	-\frac{1}{z}\left(\cfrac{s_{0,0}}{\zeta}+\cfrac{s_{0,1}}{\zeta^2}\right)-\frac{s_{1,0}}{z^2\zeta}+o\left(\sum\limits_{k=0}^3\frac{1}{z^k \zeta^{3-k}}\right).
\end{equation}

Let us set
\begin{equation}\label{14.2.8}
	s_k(\zeta)=\sum\limits_{j=0}^{1-k}{j+k\choose j}\frac{s_{k,j}}{\zeta^{j+1}} .
\end{equation}
Hence, we can rewrite basic truncated problem as:

Given $\mathbf{s}(\zeta)=\{s_0(\zeta), s_1(\zeta)\}$, such that 
\begin{equation}\label{14.2.9}
s_0(\zeta)\neq0.
\end{equation}
Describe the set 
$\mathcal{M}(\mathbf{s})$ of  functions $F$, which admit the  asymptotic expansion
\begin{equation}\label{14.2.10}
	F(z,\zeta)=-\frac{s_0(\zeta)}{z}-\frac{s_1(\zeta)}{z^2}+o\left(\sum\limits_{k=0}^3\frac{1}{z^k \zeta^{3-k}}\right).\end{equation}

\begin{lemma}\label{14.L.2.1}
 Let $\mathbf{s}=\{s_{0,0}, s_{0,1}, s_{1,0}\}$,  such that \eqref{14.2.6} holds. Then any solution of the moment problem $MP(\mathbf{s})$  admits the  representation
 \begin{equation}\label{14.2.11}
	F(z,\zeta)=-\cfrac{s_{0,0}}{z\zeta-\cfrac{s_{1,0}}{\cfrac{s_{0,0}}{\zeta}+\cfrac{s_{0,1}}{\zeta^2}}+\zeta\tau(z,\zeta)},
\end{equation}
where the parameter $\tau$ satisfies 
 \begin{equation}\label{14.2.12}
	\tau(z,\zeta)=o\left(\cfrac{s_1(\zeta)}{s_0(\zeta)}\right).
\end{equation}
\end{lemma}
\begin{proof}
Assume $F$ admits the  asymptotic expression \eqref{14.2.7} and \eqref{14.2.6} holds. By \eqref{14.2.8}, we can rewrite  \eqref{14.2.7} and  \eqref{14.2.10} is obtained. Hence, we reduce  the moment problem $MP(\mathbf{s})$ to  one-dimensional problem. Consequently,  any solution of $MP(\mathbf{s})$ takes the form
\begin{equation}\label{14.2.14}\begin{split}
 	F(z,\zeta)&=-\cfrac{b_0(\zeta)}{z-\cfrac{s_1(\zeta)}{s_0(\zeta)}+\tau(z,\zeta)}=-\cfrac{\cfrac{s_{0,0}}{\zeta}+\cfrac{s_{0,1}}{\zeta^2}}{z-\cfrac{\cfrac{s_{1,0}}{\zeta}}{\cfrac{s_{0,0}}{\zeta}+\cfrac{s_{1,0}}{\zeta^2}}+\tau(z,\zeta)}=\\&=-\cfrac{\cfrac{s_{0,0}}{\zeta}}{z-\cfrac{\cfrac{s_{1,0}}{\zeta}}{\cfrac{s_{0,0}}{\zeta}+\cfrac{s_{1,0}}{\zeta^2}}+\tau(z,\zeta)}
	-\cfrac{\cfrac{s_{0,1}}{\zeta^2}}{z-\cfrac{\cfrac{s_{1,0}}{\zeta}}{\cfrac{s_{0,0}}{\zeta}+\cfrac{s_{1,0}}{\zeta^2}}+\tau(z,\zeta)}.
\end{split}\end{equation}
We simplify \eqref{14.2.14} and obtain \eqref{14.2.11}, where $b_0(\zeta)=\frac{s_{0,0}}{\zeta}+\frac{s_{0,1}}{\zeta^2} $ and the parameter  $\tau$ satisfies \eqref{14.2.12}. This completes the proof.
\end{proof}

\subsection{Truncated   problem $MP(\mathbf{s},\mathfrak{N}(\mathbf{s}))$ }

\textbf{Problem $MP(\mathbf{s},\mathfrak{N}(\mathbf{s}))$}: Given a sequence $\mathbf{s}=\{s_{i,j}\}_{i,j=0}^{i+j= 2n-1}$, $n\in\mathbb{N}$. Describe the set 
$\mathcal{M}(\mathbf{s})$ of  functions $F$, which have the  asymptotic expansion
\begin{equation}\label{14.2.15}\begin{split}
	F(z,\zeta)&=-\frac{1}{z}\sum_{j=0}^{2n-1}\frac{s_{0,j}}{\zeta^{j+1}}-\frac{1}{z^2}\sum_{j=0}^{2n-2}\frac{(j+1)s_{1,j}}{\zeta^{j+1}}-\ldots
    -\\&-\frac{1}{z^{2n-1}}\sum_{j=0}^{1}{2n+j-2\choose j}\frac{s_{2n-2,j}}{\zeta^{j+1}}-\frac{s_{2n-1,0}}{z^{2n}\zeta}+o\left(\sum\limits_{k=0}^{2n+1}\frac{1}{z^k \zeta^{n-k}}\right).\end{split}
\end{equation}
In additional, if we put
\begin{equation}\label{14.2.16}
	s_k(\zeta)=\sum\limits_{j=0}^{2n-1-k}{j+k\choose j}\frac{s_{k,j}}{\zeta^{j+1}}\quad\mbox{for all }\quad k=\overline{0,2n-1},
\end{equation}
then we obtain the new sequence $\mathbf{s}(\zeta)=\{s_i\}_{i=0}^{2n-1}$, $\mathbf{s}(\zeta)$ is called an associated moment sequence with $\mathbf{s}$.

Moreover,  \eqref{14.2.15}   can be rewritten as 
\begin{equation}\label{14.2.17}
	F(z,\zeta)=-\frac{s_0(\zeta)}{z}-\frac{s_1(\zeta)}{z^2}-\ldots--\frac{s_{2n-1}(\zeta)}{z^{2n}}+o\left(\sum\limits_{k=0}^{2n+1}\frac{1}{z^k \zeta^{n-k}}\right).
\end{equation}

\begin{theorem}\label{14.2.th.2.1}
Let $\mathbf{s}=\{s_{i,j}\}_{i,j=0}^{i+j= 2n-1}$ be a sequence of real numbers and let   $\mathbf{s}(\zeta)=\{s_k\}_{k=0}^{2n-1}$ be the associated moment sequence of $\mathbf{s}$, which is  defined by~\eqref{14.2.16} and $\mathbf{s}(\zeta)$ is the simple regular sequence on the set  $\mathfrak{N}(\mathbf{s})$. Then any solution of the  problem $MP(\mathbf{s},\mathfrak{N}(\mathbf{s}))$  takes the representation
\begin{equation}\label{14.2.19}
 	F(z,\zeta)=-\cfrac{b_0(\zeta)}{z-\cfrac{s_1(\zeta)}{s_0(\zeta)}-\cfrac{b_1(\zeta)}{z-\cfrac{s_1^{(1)}(\zeta)}{s_0^{(1)}(\zeta)}-\ldots-\cfrac{b_{n-1}(\zeta)}{z-\cfrac{s_1^{(n-1)}(\zeta)}{s_0^{(n-1)}(\zeta)}+\tau(z,\zeta)}}},
\end{equation}
where the parameter $\tau$ satisfies 
 \begin{equation}\label{14.2.20}
	\tau(z,\zeta)=o\left(\cfrac{s_1^{(n-1)}(\zeta)}{s_0^{(n-1)}(\zeta)}\right),
\end{equation}
the recursive sequences $\mathbf{s}^{(i)}(\zeta)=\left\{s_j^{(i)}(\zeta)\right\}_{j=0}^{2n-2i-1}$ are defined by
 \begin{equation}\label{14.2.21}
s_{j}^{(i)}(\zeta)=\cfrac{(-1)^{j+1}}{\left(s_0^{(i-1)}(\zeta)\right)^{j+3}}\begin{vmatrix}
s_1^{(i-1)}(\zeta) & s_0^{(i-1)}(\zeta) &0&\ldots& 0\\
\vdots& \ddots &\ddots&\ddots & \vdots\\
\vdots&  &\ddots&\ddots &0\\
\vdots&  & &\ddots & s_0^{(i-1)}(\zeta)\\
s_{j+2}^{(i-1)} (\zeta)& \ldots &\ldots&\ldots& s_1^{(i-1)} (\zeta)
\end{vmatrix},
\end{equation}
where $s_j^{(0)}(\zeta) =s_j(\zeta)$ and $i=\overline{1,2n-2i-1}$, $b_i$ can be calculated by
 \begin{equation}\label{14.2.21x}
 b_{0}(\zeta)=s_{0}(\zeta)\quad\mbox{and}\quad b_{i}(\zeta)=s_{0}^{(i)}(\zeta).
\end{equation}
\end{theorem}
\begin{proof}
Suppose the sequence  $\mathbf{s}(\zeta)=\{s_j(\zeta)\}_{j=0}^{2n-1}$ is the associated  with $\mathbf{s}=\{s_{i,j}\}_{i,j=0}^{i+j=2n-1}$, $s_j(\zeta)$ are defined by~\eqref{14.2.16} and  $\mathbf{s}(\zeta)$ is the simple regular sequence on the set  $\mathfrak{N}(\mathbf{s})$. Let  $F$ admit  the asymptotic expansion
\[
	F(z,\zeta)=-\frac{s_0(\zeta)}{z}-\frac{s_1(\zeta)}{z^2}-\ldots--\frac{s_{2n-1}(\zeta)}{z^{2n}}+o\left(\sum\limits_{k=0}^{2n+1}\frac{1}{z^k \zeta^{n-k}}\right).
\]
By Lemma~\ref{14.L.2.1} and a similar  Schur algorithm (see \cite{DK17}, \cite{K17}), we obtain 
\[
	F(z,\zeta)=-\cfrac{b_0(\zeta)}{z-\cfrac{s_1(\zeta)}{s_0(\zeta)}+F_1(z,\zeta)}, 
\]
where $b_0(\zeta)=s_{0}(\zeta)$  and  $F_1(z,\zeta)$ admits the asymptotic
\[
	F_1(z,\zeta)=-\frac{s^{(1)}_0(\zeta)}{z}-\frac{s^{(1)}_1(\zeta)}{z^2}-\ldots--\frac{s^{(1)}_{2n-3}(\zeta)}{z^{2n-2}}+o\left(\sum\limits_{k=0}^{2n-3}\frac{s^{(1)}_k(\zeta)}{z^{k+1}}\right)
\]
and the recursive sequence $\mathbf{s}^{(1)}(\zeta)=\left\{s_j^{(1)}(\zeta)\right\}_{j=0}^{2n-3}$ is defined by~\eqref{14.2.21}.

By induction,   we get  the   representation  
\[
	F_1(z,\zeta)=-\cfrac{b_1(\zeta)}{z-\cfrac{s_1^{(1)}(\zeta)}{s_0^{(1)}(\zeta)}+F_2(z,\zeta)}, 
\]
where $b_1(\zeta)=s_{0}^{(1)}(\zeta)$  and the  function $F_2(z,\zeta)$ has the  asymptotic expansion 
\[
	F_2(z,\zeta)=-\frac{s^{(2)}_0(\zeta)}{z}-\frac{s^{(2)}_1(\zeta)}{z^2}-\ldots--\frac{s^{(2)}_{2n-5}(\zeta)}{z^{2n-4}}+o\left(\sum\limits_{k=0}^{2n-5}\frac{s^{(2)}_k(\zeta)}{z^{k+1}}\right),
\]
the recursive sequence $\mathbf{s}^{(2)}(\zeta)=\left\{s_j^{(2)}(\zeta)\right\}_{j=0}^{2n-5}$ can be calculated by~\eqref{14.2.21}.

Continuing this process $N-2$ times, we obtaine \eqref{14.2.19}--\eqref{14.2.21x}. This completes the proof.
\end{proof}

\section{Symmetric form}

Now, we study the truncated problem $MP(\mathbf{s},\mathfrak{N}(\mathbf{s}))$ and find a description of  all solutions in the symmetric form. In particular, this is more comfortable for   spherical problems.

Let $F$ be the associated function defined by~\eqref{14.2.2}. Then $F$ can be rewritten in the symmetric form as 
\begin{equation}\label{14.2.3sym}\begin{split}
	F(x, \zeta)=&-\sum_{i,j=0}^\infty{i+j\choose i} s_{i,j}\frac{1}{z^{i+1}}\frac{1}{\zeta^{j+1}}=-\\&-\cfrac{s_{0,0}}{z\zeta}-\cfrac{s_{0,1}z+2s_{1,1}+s_{1,0}\zeta}{z^2\zeta^2}-\\&-
	\cfrac{\sum\limits_{i=0}^{1}{i+2\choose i}s_{i,2}z^{2-i}+{2+2\choose 2}s_{2,2}+\sum\limits_{i=0}^{1}{2+i\choose 2}s_{2,i}\zeta^{2-i}}{z^3\zeta^3}-\cdots-\\&-
	\cfrac{\sum\limits_{i=0}^{n-1}{i+n\choose i}s_{i,n}z^{n-i}+{n+n\choose n}s_{n,n}+\sum\limits_{i=0}^{n-1}{n+i\choose n}s_{n,i}\zeta^{n-i}}{z^{n+1}\zeta^{n+1}}-\cdots
\end{split}
\end{equation}


First of all, we consider the basic problems, i.e. we construct the first step of the solutions with the minimum data of the sequence $\mathbf{s}$.  After that, we study  the general case.

\subsection{Basic 1st type}
The first basic problem can be formulated as:

Given a sequence of real numbers $\mathbf{s}=\{s_{0,0}, s_{0,1}, s_{1,0}\}$. Find  all function $F$, which admit the   asymptotic expansion
 \begin{equation}\label{14.3.1.X1}
	F(z,\zeta)=-\cfrac{s_{0,0}}{z\zeta}-\cfrac{s_{0,1}z+s_{1,0}\zeta}{z^2\zeta^2}+O\left(\cfrac{1}{z^2\zeta^2}\right).
   \end{equation}
\begin{proposition}\label{14p.sym. pro.b1} Let $\mathbf{s}=\{s_{0,0}, s_{0,1}, s_{1,0}\}$ be a  sequence of real numbers and let $s_{0,0}~\neq~0$.  Then any solution of the  problem $MP(\mathbf{s},\mathfrak{N}(\mathbf{s}))$  admits the representation
 \begin{equation}\label{14.3.1}
 	F(z,\zeta)=-\cfrac{b_0}{a_0(z,\zeta)+\tau(z,\zeta)},
 \end{equation}
 where
  \begin{equation}\label{14.3.2}
  a_0(z,\zeta)=z\zeta-\cfrac{s_{0,1}z+s_{1,0}\zeta}{s_{0,0}},\quad b_0=s_{0,0},\quad\mbox{and}\quad \tau(z,\zeta)=O\left(1\right).
   \end{equation}
\end{proposition}
\begin{proof} Let $s_{0,0}\neq 0$ and $F$ admit the  asymptotic expansion
  \begin{equation}\label{14.3.3}
	F(z,\zeta)=-\cfrac{s_{0,0}}{z\zeta}-\cfrac{s_{0,1}}{z\zeta^2}-\cfrac{s_{1,0}}{z^2\zeta}+O\left(\cfrac{1}{z^2\zeta^2}\right).
	   \end{equation}
Consequently, \eqref{14.3.3} can be  rewritten as
  \begin{equation}\label{14.3.4}
	-\cfrac{s_{0,0}}{z\zeta}-\cfrac{s_{0,1}z+s_{1,0}\zeta}{z^2\zeta^2}+O\left(\cfrac{1}{z^2\zeta^2}\right).
	   \end{equation}
Setting
 \begin{equation}\label{14.3.4*}
 \mathfrak{s}_0(z,\zeta)=s_{0,0} \quad\mbox{and}\quad  \mathfrak{s}_1(z,\zeta)=s_{0,1}z+s_{1,0}\zeta,
\end{equation}
 then \eqref{14.3.4} can be represented as
\begin{equation}\label{14.3.5}
  	F(z,\zeta)-\cfrac{\mathfrak s_{0}(z,\zeta)}{z\zeta}-\cfrac{\mathfrak s_{1}(z,\zeta)}{(z\zeta)^2}+O\left(\cfrac{1}{(z\zeta)^2}\right).
\end{equation} 
 By Theorem~\ref{14p.int.th_ind}, we obtain
\begin{equation}\label{14.3.6}
-\cfrac{b_0}{a_0(z,\zeta)+\tau(z,\zeta)},
 \end{equation} 
  where the parameter $\tau(z,\zeta)=O\left(1\right)$ and  the atom $(a_0,b_0)$ is defined by
 \begin{equation}\label{14.3.7}
  b_{0}=\mathfrak s_{0}(z,\zeta)\quad\mbox{and}\quad a_{0}(z,\zeta)=z\zeta-\cfrac{\mathfrak s_{1}(z,\zeta)}{\mathfrak s_{0}(z,\zeta)}.
   \end{equation} 
Substituting~\eqref{14.3.4*} into~\eqref{14.3.6}--\eqref{14.3.7}, we obtain~\eqref{14.3.1}--\eqref{14.3.2}. This completes the proof.
~\end{proof}

\subsection{Basic 2nd type}

The second basic problem contains one value more than the first basic problem, one  can be formulated as follows:

Given a sequence of real numbers $\mathbf{s}=\{s_{0,0}, s_{0,1}, s_{1,0}, s_{1,1}\}$. Find  all function $F$, which admit the  asymptotic expansion
 \begin{equation}\label{14.3.1.X1qw}
	F(z,\zeta)=-\cfrac{s_{0,0}}{z\zeta}-\cfrac{s_{0,1}z+2s_{1,1}+s_{1,0}\zeta}{z^2\zeta^2}+o\left(\cfrac{1}{z^2\zeta^2}\right).
   \end{equation}
\begin{proposition}\label{14p.sym. pro.bw1} Let $\mathbf{s}=\{s_{0,0}, s_{0,1}, s_{1,0}, s_{1,1}\}$ be a sequence of real numbers and let $s_{0,0}\neq~0$.  Then any solution of the problem $MP(\mathbf{s},\mathfrak{N}(\mathbf{s}))$  admits the following representation
 \begin{equation}\label{14.3.9}
 	F(z,\zeta)=-\cfrac{b_0}{a_0(z,\zeta)+\tau(z,\zeta)},
 \end{equation}
 where
  \begin{equation}\label{14.3.10}
  a_0(z,\zeta)=z\zeta-\cfrac{s_{0,1}z+2s_{1,1}+s_{1,0}\zeta}{s_{0,0}},\quad b_0=s_{0,0}\quad\mbox{and}\quad \tau(z,\zeta)=o\left(1\right).
   \end{equation}
\end{proposition}
\begin{proof} Assume $s_{0,0}\neq 0$ and $F$ admits the asymptotic expansion
  \begin{equation}\label{14.3.11}
	F(z,\zeta)=-\cfrac{s_{0,0}}{z\zeta}-\cfrac{s_{0,1}}{z\zeta^2}-\cfrac{s_{1,0}}{z^2\zeta}+\cfrac{2s_{1,1}}{z^2\zeta^2}+o\left(\cfrac{1}{z^2\zeta^2}\right).
	   \end{equation}
Consequently, \eqref{14.3.11}  can be rewritten as
  \begin{equation}\label{14.3.12}
	F(z,\zeta)=-\cfrac{s_{0,0}}{z\zeta}-\cfrac{s_{0,1}z+2s_{1,1}+s_{1,0}\zeta}{z^2\zeta^2}+o\left(\cfrac{1}{z^2\zeta^2}\right).
	   \end{equation}
Setting
 \begin{equation}\label{14.3.13}
 \mathfrak{s}_0(z,\zeta)=s_{0,0}\quad\mbox{and}\quad \mathfrak{s}_1(z,\zeta)=s_{0,1}z+2s_{1,1}+s_{1,0}\zeta,
\end{equation}
 then \eqref{14.3.12} can be represented as
\begin{equation}\label{14.3.14}
  	F(z,\zeta)=-\cfrac{\mathfrak s_{0}(z,\zeta)}{z\zeta}-\cfrac{\mathfrak s_{1}(z,\zeta)}{z^2\zeta^2}+o\left(\cfrac{1}{z^2\zeta^2}\right).
\end{equation} 
 Again, by Theorem~\ref{14p.int.th_ind}, we obtain
\begin{equation}\label{14.3.15}
	F(z,\zeta)=-\cfrac{b_0}{a_0(z,\zeta)+\tau(z,\zeta)},
 \end{equation} 
  where the parameter $\tau(z,\zeta)=o(1)$ and the  atom $(a_0,b_0)$ is defined by
 \begin{equation}\label{14.3.16}
  b_{0}=\mathfrak s_{0}(z,\zeta)\quad\mbox{and}\quad a_{0}(z,\zeta)=z\zeta-\cfrac{\mathfrak s_{1}(z,\zeta)}{\mathfrak s_{0}(z,\zeta)}.
   \end{equation} 
Substituting~\eqref{14.3.13} into~\eqref{14.3.15}--\eqref{14.3.16}, we obtain~\eqref{14.3.9}--\eqref{14.3.10}. This completes the proof.
~\end{proof}

\subsection{General case} 

\textbf{Truncated problem }  $MP(\mathbf{s}, \mathfrak{N}(\mathbf{s}))$:
Given $n\in\mathbb{N}$ and a  sequence of real numbers $\mathbf{s}=\{s_{i,j}\}_{i,j=0}^{2n-1}$, describe the set $\mathcal{M}({\mathbf s})$  of  functions $F$, which satisfy the asymptotic expansion
\begin{equation}\label{14.3.17}\begin{split}
	F(x, \zeta)=&-\cfrac{s_{0,0}}{z\zeta}-\cfrac{s_{0,1}z+2s_{1,1}+s_{1,0}\zeta}{z^2\zeta^2}-\\&-
	\cfrac{\sum\limits_{i=0}^{1}{i+2\choose i}(s_{i,2}z^{2-i}+s_{2,i}\zeta^{2-i})+{2+2\choose 2}s_{2,2}}{z^3\zeta^3}-\cdots-\\&-
	\cfrac{\sum\limits_{i=0}^{2n-2}{i+2n\choose i}(s_{i,2n}z^{2n-1-i}+s_{2n,i}\zeta^{2n-1-i})+{4n-2\choose 2n-1}s_{2n-1,2n-1}}{z^{2n}\zeta^{2n}}\,+\\&+o\left(\cfrac{1}{z^{2n}\zeta^{2n}}\right).\end{split}
\end{equation}
Setting
\begin{equation}\label{14p.3.24}\begin{split}
	&\mathfrak s_j^{(0)}(z,\zeta)=\sum\limits_{i=0}^{j-1}{i+j\choose i}\left(s_{i,j}z^{j-i}+s_{j,i}\zeta^{j-i}\right)+{2j\choose j}s_{j,j},
\\&\mathfrak s_0^{(0)}(z,\zeta)=s_{0,0},\quad 
	j=\overline{1,2n-1},
	\end{split}
\end{equation}
we obtain the new sequence  $\mathfrak {s}^{(0)}(z,\zeta)=\{\mathfrak {s}^{(0)}_{j}(z,\zeta)\}_{j=0}^{2n-1}$  associated with  the sequence   $\mathbf{s}$ and $\mathfrak {s}^{(0)}(z,\zeta)$ is the simple regular sequence on the set $\mathfrak{N}(\mathbf{s})$. Moreover, \eqref{14.3.17} can be rewritten as
\begin{equation}\label{14p.3.27}
 	F(z,\zeta)=-\cfrac{\mathfrak{s}_0^{(0)} (z,\zeta)}{z\zeta}-\cfrac{\mathfrak{s}_1^{(0)}(z,\zeta) }{z^2\zeta^2 }-\cdots--\cfrac{\mathfrak{s}^{(0)}_{2n-1} (z,\zeta)}{z^{2n}\zeta^{2n} }+O\left(\cfrac{1}{z^{2n}\zeta^{2n} }\right).
\end{equation}

\begin{theorem}\label{14p.sym. pro.th1} Let  $\mathbf{s}=\{s_{i,j}\}_{i,j=0}^{2n-1}$ be a sequence of real numbers and let $\mathfrak {s}(z,\zeta)=\{\mathfrak {s}_{j}(z,\zeta)\}_{j=0}^{2n-1}$ be an associated sequence of $\mathbf{s}$, which is defined by~\eqref{14p.3.24} and $\mathfrak {s}(z,\zeta)$ is a simple regular sequence on the set $\mathfrak{N}(\mathbf{s})$. Then any solution of the truncated problem $MP(\mathbf{s}, \mathfrak{N}(\mathbf{s}))$  admits the   representation
\begin{equation}\label{14p.3.21}
	F(z,\zeta)=-\cfrac{b_0}{a_0(z, \zeta)-\cfrac{b_1(z, \zeta)}{a_1(z, \zeta)-\cdots-\cfrac{b_{n-1}(z, \zeta)}{a_{n-1}(z, \zeta)+\tau(z, \zeta)}}}, 
\end{equation}
where the parameter $\tau$ satisfies the condition
\begin{equation}\label{14p.3.22}
	\tau(z,\zeta)=o\left(\cfrac{\mathfrak s^{(n-1)}_1(z,\zeta)}{\mathfrak s_{0}^{(n-1)}(z,\zeta)}\right), 
\end{equation}
the atoms $(a_j,b_j)$ are defined by
\begin{equation}\label{14p.3.23}\begin{split} &
	a_{0}(z,\zeta)=z\zeta-\cfrac{\mathfrak s_1^{(0)}(z,\zeta)}{\mathfrak s_0^{(0)}(z,\zeta)},\quad b_0=\mathfrak s_0^{(0)}(z,\zeta),\\&a_j(z,\zeta)=z\zeta-\cfrac{\mathfrak s^{(j)}_1(z,\zeta)}{\mathfrak s_{0}^{(j)}(z,\zeta)}, \quad b_j(z,\zeta)=\mathfrak s_{0}^{(j)}(z,\zeta),\quad j=\overline{0,n-1},
	\end{split}
\end{equation}
the recursive sequences ${\bf\mathfrak s}^{(j)}(z,\zeta)=\left\{\mathfrak s^{(j)}_i(z,\zeta)\right\}_{i=0}^{2n-1-2j}$ are defined for all $i=\overline{0,2n-1-2j}$ 
 \begin{equation}\label{14p.3.25}
\mathfrak s_{i}^{(j)}\!(z,\zeta)\!=\!\cfrac{(-1)^{i+1}}{\left(\mathfrak s_0^{(j-1)}(z,\zeta)\right)^{i+3}}\!\!\begin{vmatrix}
\mathfrak s_1^{(j-1)}\!(z,\zeta) & \mathfrak s_0^{(j-1)} \!(z,\zeta)&0&\ldots& 0\\
\vdots& \ddots &\ddots&\ddots & \vdots\\
\vdots&  &\ddots&\ddots &0\\
\vdots&  & &\ddots & \mathfrak s_0^{(j-1)}\!(z,\zeta) \\
\mathfrak s_{i+2}^{(i-1)} \!(z,\zeta)& \ldots &\ldots&\ldots& \mathfrak s_1^{(j-1)} \!(z,\zeta)
\end{vmatrix}.
\end{equation}
\end{theorem}
\begin{proof}
Let   $\mathbf{s}=\{s_{i,j}\}_{i,j=0}^{2n-1}$ be a sequence of real numbers and let $\mathfrak{s}^{(0)}(z,\zeta)=\{\mathfrak{s}^{(0)}_{i}(z,\zeta)\}_{i=0}^{2n-1}$ be  associated sequence of $\mathbf{s}$, which is defined by \eqref{14p.3.24} and  the simple regular on the set  $\mathfrak{N}(\mathbf{s})$.
Suppose the associated function  F admits the asymptotic expansion~\eqref{14.3.17}. 

By Theorem~\ref{14p.int.th_ind}, the associated function $F$ can be represented by
\begin{equation}\label{14p.3.28}
	F(z,\zeta)=-\cfrac{b_0}{a_0(z,\zeta)+F_{1}(z,\zeta)}, 
\end{equation}
where the atom $(a_0, b_0)$ can be calculated by
\[
	a_{0}(z,\zeta)=z\zeta-\cfrac{\mathfrak s_1^{(0)}(z,\zeta)}{\mathfrak s_0^{(0)}(z,\zeta)}\quad\mbox{and}\quad b_0=\mathfrak s_0^{(0)}(z,\zeta).
\]
We associate a  function $F_1$ with the sequence $\mathfrak{s}^{(1)}(z,\zeta)=\left\{\mathfrak{s}^{(1)}_{i}(z,\zeta)\right\}_{i=0}^{2n-3}$ defined by \eqref{14p.3.25} and obtain
\begin{equation}\label{14p.3.27&*}
 	F_1(z,\zeta)=-\cfrac{\mathfrak{s}_0^{(1)} (z,\zeta)}{z\zeta}-\cdots-\cfrac{\mathfrak{s}^{(1)}_{2n-3} (z,\zeta)}{z^{2n-2}\zeta^{2n-2} }+o\left(\sum_{k=0}^{2n-3}\cfrac{\mathfrak{s}^{(1)}_{k} (z,\zeta)}{z^{k+1}\zeta^{k+1} }\right).
\end{equation}

By induction on $j$, we obtain that  any solution of the problem $MP(\mathbf{s}, \mathfrak{N}(\mathbf{s}))$  takes the representation~\eqref{14p.3.21}--\eqref{14p.3.25}. This completes the proof.
~\end{proof}

 
\subsection{The $n$-th convergent}
 Now, we study the continued fraction of several variables  and its $n$-th convergent. Let the continued fraction be defined by 
 \begin{equation}\label{14p.4*.1}
	-\cfrac{b_0(z_1,\ldots,z_k)}{a_0(z_1,\ldots,z_k)-\cfrac{b_1(z_1,\ldots,z_k)}{a_1(z_1,\ldots,z_k)-\ddots}}.
\end{equation}

Moreover, let  us define a linear fractional transformation $T_j$ by 
 \begin{equation}\label{14p.4*.2}
	T_j(\omega)=-\cfrac{b_j(z_1,\ldots,z_k)}{a_j(z_1,\ldots,z_k)+\omega(z_1,\ldots,z_k)}, \quad j\in\mathbb{Z}.
\end{equation}

Similar to \cite{Wall}, the $n$--th convergent of the fraction \eqref{14p.4*.1} can be represented by 
 \begin{equation}\label{14p.4*.3}
 	\cfrac{Q_n(z_1,\ldots,z_k)}{P_n(z_1,\ldots,z_k)}=T_0\circ T_1\circ \ldots \circ T_n(0),
 \end{equation}
where the numerator  $Q_n$ and denominator $P_n$ are the solutions of the system 
 \begin{equation}\label{14p.4*.4}
	y_{j+1}=a_j(z_1,\ldots,z_k)y_j-b_j(z_1,\ldots,z_k)y_{j-1}
 \end{equation}
 subject  to the initial conditions 
  \begin{equation}\label{14p.4*.5}\begin{split}&
  P_{-1}(z_1,\ldots,z_k)\equiv 0\quad\mbox{and}\quad P_0(z_1,\ldots,z_k)\equiv 1,\\&
  Q_{-1}(z_1,\ldots,z_k)\equiv 1\quad\mbox{and}\quad Q_0(z_1,\ldots,z_k)\equiv 0.
  \end{split} \end{equation}
\begin{proposition}\label{14p.sym. prop4.1*} Let $\mathbf{s}=\{s_{i,j}\}_{i,j=0}^{2n-1}$ be a sequence of real numbers and let $\mathfrak {s}(z,\zeta)=\{\mathfrak {s}_{j}(z,\zeta)\}_{j=0}^{2n-1}$ be an associated sequence of $\mathbf{s}$, which is defined by~\eqref{14p.3.24} and $\mathfrak {s}$ is  a simple regular sequence on the set $\mathfrak{N}(\mathbf{s})$. Then any solution of the truncated problem $MP(\mathbf{s}, \mathfrak{N}(\mathbf{s}))$  admits the representation
  \begin{equation}\label{14p.4*.6}
  	F(z,\zeta)=\cfrac{Q_{n-2}(z,\zeta)\tau(z,\zeta)+Q_{n-1}(z,\zeta)}{P_{n-2}(z,\zeta)\tau(z,\zeta)+P_{n-1}(z,\zeta)},
 \end{equation}
where the parameter $\tau$ satisfies \eqref{14p.3.22}, $P_j$ and $Q_j$ are the solutions of the   system 
 \begin{equation}\label{14p.4*.7}
	y_{j+1}=a_j(z,\zeta)y_j-b_j(z,\zeta)y_{j-1}
 \end{equation}
 subject  to the  initial conditions 
  \begin{equation}\label{14p.4*.8}
  P_{-1}(z,\zeta)\equiv 0\quad,P_0(z,\zeta)\equiv 1,\quad
  Q_{-1}(z,\zeta)\equiv 1\quad\mbox{and}\quad Q_0(z,\zeta)\equiv 0,
\end{equation}
and the atoms $(a_j, b_j)$ are defined by \eqref{14p.3.23}.
\end{proposition}
\begin{proof}
  By Theorem~\ref{14p.sym. pro.th1}, any solution $F$ of the problem $MP(\mathbf{s}, \mathfrak{N}(\mathbf{s}))$ can be represented by \eqref{14p.3.21}--\eqref{14p.3.25}. According to \eqref{14p.4*.2} and \eqref{14p.4*.3}, we  define  the  linear fractional transformation $T_j$ by 
  \begin{equation}\label{14p.4*.9}
  	T_j(\tau)=-\cfrac{b_j(z,\zeta)}{a_j(z,\zeta)+\tau(z,\zeta)}.
  \end{equation}
  Hence, the representation \eqref{14p.3.21} of $F$ can be rewritten as 
    \begin{equation}\label{14p.4*.10}
    	F(z,\zeta)=T_0\circ T_1\circ \ldots \circ T_{n-1}(\tau)=\cfrac{Q_{n-2}(z,\zeta)\tau(z,\zeta)+Q_{n-1}(z,\zeta)}{P_{n-2}(z,\zeta)\tau(z,\zeta)+P_{n-1}(z,\zeta)},
      \end{equation}
 where  $\tau$   satisfies \eqref{14p.3.22}, $P_j$ and $Q_j$ are defined by \eqref{14p.4*.7}--\eqref{14p.4*.8}.  This completes the proof.~\end{proof}
\section{Stieltjes like case}
In this section, we study the truncared two-dimensional moment problem and find description of all solutions in terms of Stieltjes like fractions.

\begin{definition} Let $\mathbf{s}=\{s_j(z_1,\ldots, z_n)\}_{j=0}^{\ell}$ be a sequence such that $s_j(z_1,\ldots, z_n)$ depend of the variables $z_1,\ldots, z_n$. The set $\mathfrak{N}^+(\mathbf{s})$ is called a simple  plus regular set of the sequence $\mathbf{s}$ and defined by
\begin{equation}\label{14.7.1}
	\mathfrak{N}^+(\mathbf{s})=\{(z_1,\ldots, z_n) |\, D_i(z_1,\ldots, z_n) \neq 0\mbox{ and } D^+_i(z_1,\ldots, z_n) \neq 0,\quad i=\overline{0,[\ell/2]}  \},
\end{equation}
where $D_i$ are defined by \eqref{14.2.18} and
\begin{equation}\label{14.7.2}
	D^+_i(z_1,\ldots, z_n) =\begin{vmatrix}
s_1(z_1,\ldots, z_n)& \ldots& s_{i}(z_1,\ldots, z_n)\\
  \vdots &\ddots &\vdots \\
s_{i}(z_1,\ldots, z_n) & \ldots  & s_{2i-1}(z_1,\ldots, z_n)
\end{vmatrix}.
\end{equation}

In this case,  $\mathbf{s}=\{s_j(z_1,\ldots, z_n)\}_{j=0}^{\ell}$ is called a  simple  plus regular sequence on the set  $\mathfrak{N}^+(\mathbf{s})$.
\end{definition}
\subsection{Basic problem} First of all, we study the basic problem $MP(\mathbf{s}, \mathfrak{N}^+(\mathbf{s}))$ formulated as:
Given the sequence of real number  $\mathbf{s}=\{s_{0,0}, s_{0,1},s_{1,0}, s_{1,1}\}$, such that the associated sequence $\mathfrak {s}(z,\zeta)=\{\mathfrak {s}_{0}(z,\zeta), \mathfrak {s}_{1}(z,\zeta)\}$ is a simple plus regular on the set $\mathfrak{N}^{+}(\mathbf{s})$ and 
\begin{equation}\label{14p.7.1.i1}
	 \mathfrak {s}_{0}(z,\zeta)=s_{0,0}\quad\mbox{and}\quad s_1(z,\zeta)=s_{0,1}z+2s_{1,1}+s_{1,0}\zeta.
\end{equation}
 Find  all functions $F$, which admit the asymptotic expansion
 \begin{equation}\label{14.3.1.tX1}\begin{split}
	F(z,\zeta)&=-\cfrac{s_{0,0}}{z\zeta}-\cfrac{s_{0,1}z+2s_{1,1}+s_{1,0}\zeta}{z^2\zeta^2}+o\left(\cfrac{1}{z^2\zeta^2}\right)=\\&=
	-\cfrac{ \mathfrak {s}_{0}(z,\zeta)}{z\zeta}-\cfrac{ \mathfrak {s}_{1}(z,\zeta)}{z^2\zeta^2}+o\left(\cfrac{1}{z^2\zeta^2}\right).
  \end{split} \end{equation}

\begin{lemma}\label{14p.lem7.1} Let  $\mathbf{s}=\{s_{0,0}, s_{0,1},s_{1,0}, s_{1,1}\}$  be a sequence of real numbers and let $\mathfrak {s}(z,\zeta)=\{\mathfrak {s}_{0}(z,\zeta), \mathfrak {s}_{1}(z,\zeta)\}$ be an associated sequence of $\mathbf{s}$, which is defined by \eqref{14p.7.1.i1}and $\mathfrak {s}(z,\zeta)$ is a  simple plus regular sequence on the set $\mathfrak{N}^{+}(\mathbf{s})$. Then any solution of the  basic problem $MP(\mathbf{s}, \mathfrak{N}^+(\mathbf{s}))$  takes the  form
\begin{equation}\label{14p.7.3r1}
	F(z,\zeta)=\cfrac{1}{-z\zeta m_1(z,\zeta)+\cfrac{1}{l_1(z,\zeta)+\tau(z,\zeta)}},
\end{equation}
where the parameter $\tau$ satisfies the condition
\begin{equation}\label{14p.7.4r.2}
	\tau(z,\zeta)=o\left( l_1(z,\zeta)\right)
\end{equation}
and the atom $(m_1,l_1)$ can be calculated by 
\begin{equation}\label{14p.7.4r.3}
	m_1(z,\zeta)=\cfrac{1}{s_{0,0}}\quad \mbox{and}\quad l_1(z,\zeta)=\cfrac{s_{0,0}^2}{s_{0,1}z+2s_{1,1}+s_{1,0}\zeta}.
\end{equation}
\end{lemma}
\begin{proof} Let the assumption of  Theorem hold.  Consequently, any solution of the  basic problem $MP(\mathbf{s}, \mathfrak{N}^+(\mathbf{s}))$ admits the  expansion 
\[\begin{split}
	F(z,\zeta)&=-\cfrac{s_{0,0}}{z\zeta}-\cfrac{s_{0,1}z+2s_{1,1}+s_{1,0}\zeta}{z^2\zeta^2}+o\left(\cfrac{1}{z^2\zeta^2}\right)=\\&=
	-\cfrac{\mathfrak {s}_0(z,\zeta)}{z\zeta}-\cfrac{\mathfrak {s}_1(z,\zeta)}{z^2\zeta^2}+o\left(\cfrac{1}{z^2\zeta^2}\right).
	\end{split}
\]
By Proposition \ref{14p.sym. pro.bw1}, we obtain that $F$ can be represented by
\begin{equation}\label{14p.7.4r.5}
	F(z,\zeta)=-\cfrac{b_0(z,\zeta)}{a_0(z,\zeta)+\tau_1(z,\zeta)}=-\cfrac{\mathfrak {s}_0(z,\zeta)}{z\zeta-\cfrac{\mathfrak {s}_1(z,\zeta)}{\mathfrak {s}_0(z,\zeta)}+\tau_1(z,\zeta)},
\end{equation}
where $\tau_1(z,\zeta)=o(1)$. On the other hand, we can rewrite \eqref{14p.7.4r.5} as follows
\[\begin{split}
	F(z,\zeta)&=\cfrac{1}{-\cfrac{z\zeta}{\mathfrak {s}_0(z,\zeta)}+\cfrac{\mathfrak {s}_1(z,\zeta)}{\mathfrak {s}^2_0(z,\zeta)}-\cfrac{\tau_1(z,\zeta)}{\mathfrak {s}_0(z,\zeta)}}=\\&=
	\left\{\mbox{Due to }s_0(z,\zeta)=s_{0,0},\mbox{then } \cfrac{\tau_1(z,\zeta)}{\mathfrak {s}_0(z,\zeta)}=o(1)\right\}=\\&=
	\cfrac{1}{-\cfrac{z\zeta}{\mathfrak {s}_0(z,\zeta)}+\cfrac{1}{\cfrac{\mathfrak {s}^2_0(z,\zeta)}{\mathfrak {s}_1(z,\zeta)}+\tau(z,\zeta)}}=\\&
	=\cfrac{1}{-\cfrac{z\zeta}{s_{0,0}}+\cfrac{1}{\cfrac{s^2_{0,0}}{s_{0,1}z+2s_{1,1}+s_{1,0}\zeta}+\tau(z,\zeta)}}.
\end{split}
\]
Hence, the atom $(m_1,l_1)$ can be defined by \eqref{14p.7.4r.3} and the parameter   $\tau(z,\zeta)$ satisfies \eqref{14p.7.4r.2}. This completes the proof.
\end{proof}
\subsection{General case}
 \textbf{Truncated problem $MP(\mathbf{s}, \mathfrak{N}^+(\mathbf{s}))$}:
 Given $n\in\mathbb{N}$ and  a sequence $\mathbf{s}=\{s_{i,j}\}_{i,j=0}^{2n-1}$ of real numbers, describe the set $\mathcal{M}^+({\mathbf s})$  of  functions $F$, which satisfy the asymptotic expansion
\begin{equation}\label{14.3.40r}\begin{split}
	F(x, \zeta)=&-\cfrac{s_{0,0}}{z\zeta}-\cfrac{s_{0,1}z+2s_{1,1}+s_{1,0}\zeta}{z^2\zeta^2}-\\&-
	\cfrac{\sum\limits_{i=0}^{1}{i+2\choose i}(s_{i,2}z^{2-i}+s_{2,i}\zeta^{2-i})+{2+2\choose 2}s_{2,2}}{z^3\zeta^3}-\cdots-\\&\!-\!
	\cfrac{\sum\limits_{i=0}^{2n\!-\!2}{i+2n\choose i}(s_{i,2n}z^{2n\!-\!1\!-\!i}+s_{2n,i}\zeta^{2n\!-\!1\!-\!i})\!+\!{4n\!-\!2\choose 2n\!-\!1}s_{2n\!-\!1,2n\!-\!1}}{z^{2n}\zeta^{2n}}\!+\!o\!\left(\!\cfrac{1}{z^{2n}\zeta^{2n}}\!\right)\!=\!\\&=
	-\cfrac{\mathfrak{s}_0^{(0)} (z,\zeta)}{z\zeta}-\cfrac{\mathfrak{s}_1^{(0)}(z,\zeta) }{z^2\zeta^2 }-\cdots--\cfrac{\mathfrak{s}^{(0)}_{2n-1} (z,\zeta)}{z^{2n}\zeta^{2n} }+o\left(\cfrac{1}{z^{2n}\zeta^{2n} }\right),
\end{split}
\end{equation}
where the associated sequence  $\mathfrak {s}(z,\zeta)=\{\mathfrak {s}_{j}(z,\zeta)\}_{j=0}^{2n-1}$ is the simple plus  regular sequence on the set $\mathfrak{N}^{+}(\mathbf{s})$ and defined by
\begin{equation}\label{14p.3.44r}\begin{split}&
\mathfrak s_j^{(0)}(z,\zeta)=\sum\limits_{i=0}^{j-1}{i+j\choose i}\left(s_{i,j}z^{j-i}+s_{j,i}\zeta^{j-i}\right)+{2j\choose j}s_{j,j},
\\&\mathfrak s_0^{(0)}(z,\zeta)=s_{0,0},\quad 
	j=\overline{1,2n-1}.
	\end{split}
\end{equation}

\begin{theorem}\label{14p.th7.2} Let $\mathbf{s}=\{s_{i,j}\}_{i,j=0}^{2n-1}$  be a sequence of real numbers and let $\mathfrak {s}(z,\zeta)=\{\mathfrak {s}_{j}(z,\zeta)\}_{j=0}^{2n-1}$ be an associated sequence of $\mathbf{s}$, which is defined by \eqref{14p.3.44r} and $\mathfrak {s}(z,\zeta)$ is a  simple  plus regular sequence on the set $\mathfrak{N}^{+}(\mathbf{s})$. Then any solution of the truncated problem $MP(\mathbf{s}, \mathfrak{N}^+(\mathbf{s}))$  admits the   representation
\begin{equation}\label{14p.7.3}
	F(z,\zeta)=\cfrac{1}{-z\zeta m_1(z, \zeta)+\cfrac{1}{l_1(z, \zeta)+\cdots+\cfrac{1}{-z\zeta m_{n}(z, \zeta)+\cfrac{1}{l_n(z,\zeta)+\tau(z,\zeta)}}}}, 
\end{equation}
where the parameter $\tau$ satisfies the condition
\begin{equation}\label{14p.7.4}
	\tau(z,\zeta)=o\left(l_n(z,\zeta)\right)
\end{equation}
and the atoms $(m_j,l_j)$ can be found by the atoms $(a_j,b_j)$  as 
\begin{equation}\label{14p.7.5}\begin{split}&
	b_0(z,\zeta)=\cfrac{1}{d_0(z,\zeta)}\quad\mbox{and}\quad a_0(z,\zeta)=\cfrac{1}{d_0(z,\zeta)}\left(z\zeta m_1(z,\zeta)-\cfrac{1}{l_1(z,\zeta)}\right),\\&
	b_{j-1}(z,\zeta)=\cfrac{1}{l^2_{j-1}(z,\zeta)d_{j-1}(z,\zeta)d_j(z,\zeta)},\\&
	a_{j-1}(z,\zeta)=\cfrac{1}{d_j(z,\zeta)}\left(z\zeta m_j(z,\zeta)-\cfrac{1}{l_{j-1}(z,\zeta)}-\cfrac{1}{l_j(z,\zeta)}\right).
	\end{split}
\end{equation}

\end{theorem}
\begin{proof}Let the assumption of Theorem hold. By Theorem \ref{14p.sym. pro.th1}, $F$ admits the representation \eqref{14p.3.21}--\eqref{14p.3.25}.  According to \cite[Theorem 4.1]{DK15}, we can rewrite  \eqref{14p.3.21} in terms of the Stieltjes like fraction \eqref{14p.7.3}--\eqref{14p.7.5}. This completes the proof.~\end{proof}
\subsection{The $n$-th convergent}
By \cite[Proposition 5.4]{DK17} and \cite{Wall}, the fraction \eqref{14p.7.3} is associated with the following system
 \begin{equation}\label{14p.7.7}
	\begin{cases}
  y_{2i-1}-y_{2i-3}=-z\zeta m_i(z,\zeta)y_{2i-2},\\
   y_{2i}-y_{2i-2}=l_{i}(z,\zeta) y_{2i-1}, 
 \end{cases}\quad  i=\overline{1,n}.
\end{equation}

The solutions of the system \eqref{14p.7.7} are determined by  $P_i^+$ and $Q_i^+$, which satisfy the initial conditions
 \begin{equation}\label{14p.7.8}
 	P^+_{-1}(z,\zeta)\equiv 0, \, P^+_{0}(z,\zeta)\equiv 1,\,  Q^+_{-1}(z,\zeta)\equiv 1\mbox{ and } Q_{0}^+(z,\zeta)\equiv 0.
 \end{equation}
 \begin{proposition}\label{14p.th7.2s} Let $\mathbf{s}=\{s_{i,j}\}_{i,j=0}^{2n-1}$  be a sequence of real numbers and let $\mathfrak {s}(z,\zeta)=\{\mathfrak {s}_{j}(z,\zeta)\}_{j=0}^{2n-1}$ be an associated sequence of $\mathbf{s}$, which is defined by \eqref{14p.3.44r} and $\mathfrak {s}(z,\zeta)$ is a  simple plus regular sequence on the set $\mathfrak{N}^{+}(\mathbf{s})$. Then any solution of the truncated problem $MP(\mathbf{s}, \mathfrak{N}^+(\mathbf{s}))$ takes the   representation
\begin{equation}\label{14p.7.9}
 	F(z,\zeta)=\cfrac{Q^+_{2n}(z,\zeta)+Q^+_{2n-1}(z,\zeta)\tau(z,\zeta)}{P^+_{2n}(z,\zeta)+P^+_{2n-1}(z,\zeta)\tau(z,\zeta)},
  \end{equation}
 where the parameter $\tau$ satisfies \eqref{14p.7.4}, and  $P^+_j$ and $Q^+_j$ are defined by  \eqref{14p.7.7}--\eqref{14p.7.8}.
 \end{proposition}
 \begin{proof} Assume  $F$ is the solution of the truncated problem  $MP(\mathbf{s}, \mathfrak{N}^+(\mathbf{s}))$. By Theorem~\ref{14p.th7.2}, $F$ admits the representation \eqref{14p.7.3}-\eqref{14p.7.5}. Let us define linear-fractional transformations $T_j^+$ by
 \begin{equation}\label{14p.7.10}
 	T_{2i-1}^+(\tau)=\cfrac{1}{-z\zeta m_i(z,\zeta)+\tau(z,\zeta)} \quad\mbox{and}\quad T_{2i}^+(\tau)=\cfrac{1}{l_i(z,\zeta)+\tau(z,\zeta)}.
\end{equation}
 Consequently, $F$ can be rewritten as 
  \begin{equation}\label{14p.7.11}
  	F(z,\zeta)=T_1^+\circ T^+_1\circ \ldots \circ T^+_{n-1}\circ T^+_{n}(\tau)=\cfrac{Q^+_{2n}(z,\zeta)+Q^+_{2n-1}(z,\zeta)\tau(z,\zeta)}{P^+_{2n}(z,\zeta)+P^+_{2n-1}(z,\zeta)\tau(z,\zeta)},
  \end{equation}
  where $\tau$ satisfies \eqref{14p.7.4}, $P^+_i$ and $Q_{i}^+$ are defined by \eqref{14p.7.7}--\eqref{14p.7.8}. This completes the proof.
 \end{proof}

\section{Moment problem for  atomic measures}

Let an atomic measure $\mu$ be defined by
\begin{equation}\label{14p.4.0}
	\mu(x)=\sum_{k=0}^{M}m_k\delta_{x_k}(x),
\end{equation}
where $x_k=(t_k,\tau_k)\in \mathbb{C}^2$, $m_k\in \mathbb{C}$ and $\delta$ is a Dirac delta. 

Let $\mathbf{s}=\{s_{i,j}\}_{i,j=0}^{\ell}$ be the finite moment sequence associated with the measure $\mu$, where $\ell \in  \mathbb{ Z}_{+}\cup\{\infty\} $ and
\begin{equation}\label{14p.pre4.1}
	s_{i,j}=\int_A t^i\tau^jd\mu(t,\tau)=\sum_{k=0}^M m_kt^i_k\tau^j_k.
\end{equation}

\begin{remark}
 The measure $\mu$ can be reconstructed from finitely many moments, i.e., the system \eqref{14p.pre4.1} is solvable. A sufficient condition is $\ell>M$, see e.g. \cite{KuPeRovO16}. 
\end{remark}

The associated function $F$  can be  found by
\begin{equation}\label{14p.4.1}\begin{split}F(z, \zeta)&=-\frac{1}{z\zeta}\int\int_A\frac{d\mu(t,\tau)}{1-\frac{t}{z}-\frac{\tau}{\zeta}}= -\frac{1}{z\zeta}\int\int_A\sum_{k=0}^\infty\left(\frac{t}{z}+\frac{\tau}{\zeta}\right)^k d\mu(t,\tau)=\\&=-\sum_{j=0}^M m_j\sum_{k=0}^{\infty}\left(\cfrac{t_j}{z}+\cfrac{\tau_j}{\zeta}\right)^k=-\sum_{j=0}^M m_j\sum_{k,i=0}^{\infty}{k+i\choose i} \cfrac{t_j^k}{z^{k+1}}\cfrac{\tau_j^k}{\zeta^{k+1}}=\\&=-\cfrac{\sum\limits_{j=0}^M m_j}{z\zeta}-\cfrac{\sum\limits_{j=0}^M m_j(\tau_jz+2t_i\tau_j+t_j\zeta)}{z^2\zeta^2}-\cdots-\\&-
\cfrac{\sum\limits_{j=0}^M m_j\left(\sum\limits_{i=0}^{n-1}{i+n\choose i} t_j^i\tau_j^nz^{n-i}+{2n\choose n}t_j^n\tau_j^n+ \sum\limits_{i=0}^{n-1}{n+i\choose n}t_j^n\tau_j^{i}\zeta^{n-i} \right)}{z^{n+1}\zeta^{n+1}}+\\&+o\left(\cfrac{1}{z^{n+1}\zeta^{n+1}}\right).
\end{split}
\end{equation}

We set 
\begin{equation}\label{14p.4.2}\begin{split}&
 s_k(z,\zeta)=\sum\limits_{j=0}^M m_j\left(\sum\limits_{i=0}^{k-1}{i+k\choose i} t_j^i\tau_j^kz^{k-i}+{2k\choose k}t_j^k\tau_j^k+ \sum\limits_{i=0}^{k-1}{k+i\choose k}t_j^k\tau_j^{i}\zeta^{k-i} \right),\\&
 s_0(z,\zeta)=\sum\limits_{j=0}^M m_j,\qquad  k=\overline{1,n}.
 \end{split}
\end{equation}

Hence, the function $F$ is associated with the sequence $\mathbf{s}(z,\zeta)=\{s_{i}(z,\zeta)\}_{i=0}^n$ and the sequence $\mathbf{s}$ is called the associated sequence of the measure $\mu$. Moreover, the associated function $F$ takes the short form
\begin{equation}\label{14p.4.3}
	F(z,\zeta)=-\cfrac{s_{0}(z,\zeta)}{z\zeta}-\cfrac{s_{1}(z,\zeta)}{z^2\zeta^2}-\cdots-\cfrac{s_{n}(z,\zeta)}{z^{n+1}\zeta^{n+1}}+o\left(\cfrac{1}{z^{n+1}\zeta^{n+1}}\right).
\end{equation}

\textbf{Truncated   problem $MP(\mu,2n-1,\mathfrak{N}(\mathbf{s}))$ with the atomic measure}:

Given the atomic measure $\mu$, which is defined by \eqref{14p.4.0}. Describe the set $\mathcal{M}(\mu,2n-1)$  of  functions $F$, which have the  asymptotic expansion \eqref{14p.4.1}.
\begin{theorem}\label{14p.4.th.4.1}
 Let $M\in\mathbb{Z}_{+}$, $x_j\in\mathbb{C}^2$, $j=\overline{0,M}$ and 
 $\mu(x)=\sum\limits_{j=0}^{M}m_j\delta_{x_j}(x)$ be an atomic measure and let $\mathbf{s}=\{s_{i}(z,\zeta)\}_{i=0}^{2n-1}$ be a associated sequence defined by \eqref{14p.4.2}, which is   a simple regular  on the set $\mathfrak{N}(\mathbf{s})$. Then any solution of the truncated  problem $MP(\mu,2n-1,\mathfrak{N}(\mathbf{s}))$  admits the  representation
\begin{equation}\label{14p.4.5}
	F(z,\zeta)=-\cfrac{b_0}{a_0(z, \zeta)-\cfrac{b_1(z, \zeta)}{a_1(z, \zeta)-\cdots-\cfrac{b_{n-1}(z, \zeta)}{a_{n-1}(z, \zeta)+\tau(z, \zeta)}}}, 
\end{equation}
where the parameter $\tau$ satisfies the condition
\begin{equation}\label{14p.4.6}
	\tau(z,\zeta)=o\left(\cfrac{s^{(n-1)}_1(z,\zeta)}{ s_{0}^{(n-1)}(z,\zeta)}\right)
\end{equation}
and the atoms $(a_k,b_k)$ are defined by 
\begin{equation}\label{14p.4.7}\begin{split}
	&a_{0}(z,\zeta)=z\zeta-\cfrac{s_1(z,\zeta)}{s_0(z,\zeta)},\quad b_0=s_0(z,\zeta),\\& a_k(z,\zeta)=z\zeta-\cfrac{s^{(k)}_1(z,\zeta)}{ s_{0}^{(k)}(z,\zeta)}, \quad b_k= s_{0}^{(k)}(z,\zeta),
\end{split}
\end{equation}
where the recursive sequences ${\bf s}^{(k)}(z,\zeta)=\left\{s^{(k)}_j(z,\zeta)\right\}_{j=0}^{2n-1-2k}$ are defined by 
 \begin{equation}\label{14p.4.8}
s_{j}^{(k)}(z,\zeta)=\cfrac{(-1)^{k+1}}{\left( s_0^{(k-1)}(z,\zeta)\right)^{j+3}}\begin{vmatrix}
 s_1^{(k-1)} (z,\zeta)& s_0^{(k-1)}(z,\zeta) &0&\ldots& 0\\
\vdots& \ddots &\ddots&\ddots & \vdots\\
\vdots&  &\ddots&\ddots &0\\
\vdots&  & &\ddots &s_0^{(k-1)}(z,\zeta)\\
 s_{j+2}^{(k-1)} (z,\zeta)& \ldots &\ldots&\ldots&   s_1^{(k-1)} (z,\zeta)
\end{vmatrix},\end{equation}
where  $\mathbf{s}^{(0)}(z,\zeta)= \mathbf{s}(z,\zeta)$, $ j=\overline{0,2n-1-2k}$ and $k=\overline{0,n-1}$.

Furthermore,  \eqref{14p.4.5} can be rewritten  as
 \begin{equation}\label{14p.4*.6@}
  	F(z,\zeta)=\cfrac{Q_{n-2}(z,\zeta)\tau(z,\zeta)+Q_{n-1}(z,\zeta)}{P_{n-2}(z,\zeta)\tau(z,\zeta)+P_{n-1}(z,\zeta)},
 \end{equation}
where the parameter $\tau$ satisfies \eqref{14p.4.6}, $P_j$ and $Q_j$ are the solutions of the system \eqref{14p.4*.4}-\eqref{14p.4*.5}, where the atoms $(a_j,b_j)$ are defined by \eqref{14p.4.7}

 \end{theorem}

\begin{proof}
 Let $\mu(x)=\sum\limits_{j=0}^{M}m_j\delta_{x_j}(x)$ be the atomic measure and let $\mathbf{s}(z,\zeta)=\{s_{i}(z,\zeta)\}_{i=0}^{2n-1}$ be the  associated sequence defined by \eqref{14p.4.2}, which is the simple  regular  on the set $\mathfrak{N}(\mathbf{s})$.  Then the associated function $F$ admits the asymptotic expansion
 \begin{equation}\label{14p.4.9}
	F(z,\zeta)=-\cfrac{s_{0}(z,\zeta)}{z\zeta}-\cfrac{s_{1}(z,\zeta)}{z^2\zeta^2}-\cdots-\cfrac{s_{2n-1}(z,\zeta)}{z^{2n}\zeta^{2n}}+o\left(\cfrac{1}{z^{2n}\zeta^{2n}}\right).
\end{equation}
By Theorem~\ref{14p.sym. pro.th1},  $F$ takes the form~\eqref{14p.4.5}--\eqref{14p.4.8}.

On the other hand,  by Proposition \eqref{14p.sym. prop4.1*},  \eqref{14p.4.5} can be rewritten as \eqref{14p.4*.6@} and this completes the proof.
\end{proof}
\textbf{Truncated   problem $MP(\mu,2n-1,\mathfrak{N}^+(\mathbf{s}))$ for  atomic measure}:
Now, we study a truncated problem $MP(\mu,2n-1,\mathfrak{N}^+(\mathbf{s}))$ for atomic measure. 
\begin{theorem}\label{14p.4.th.4.1@}
 Let Let $M\in\mathbb{Z}_{+}$, $x_j\in\mathbb{C}^2$, $j=\overline{0,M}$ and  $\mu(x)=\sum\limits_{j=0}^{M}m_j\delta_{x_j}(x)$ be a atomic measure and let $\mathbf{s}(z,\zeta)=\{s_{i}(z,\zeta)\}_{i=0}^{2n-1}$ be an associated sequence defined by \eqref{14p.4.2}, which is  a simple plus  regular  on the set $\mathfrak{N}^+(\mathbf{s})$. Then any solution of the truncated  problem $MP(\mu,n,\mathfrak{N}^+(\mathbf{s}))$  admits the representation
 \begin{equation}\label{14p.7.3@}
	F(z,\zeta)=\cfrac{1}{-z\zeta m_1(z, \zeta)+\cfrac{1}{l_1(z, \zeta)+\cdots+\cfrac{1}{-z\zeta m_{n}(z, \zeta)+\cfrac{1}{l_n(z,\zeta)+\tau(z,\zeta)}}}}, 
\end{equation}
where the parameter $\tau$ satisfies the condition
\begin{equation}\label{14p.7.4@}
	\tau(z,\zeta)=o\left(l_n(z,\zeta)\right)
\end{equation}
and the atoms $(m_j,l_j)$ can be found by the atoms $(a_j,b_j)$  as follows
\begin{equation}\label{14p.7.5@}\begin{split}&
	b_0(z,\zeta)=\cfrac{1}{d_0(z,\zeta)}\quad\mbox{and}\quad a_0(z,\zeta)=\cfrac{1}{d_0(z,\zeta)}\left(z\zeta m_1(z,\zeta)-\cfrac{1}{l_1(z,\zeta)}\right),\\&
	b_{j-1}(z,\zeta)=\cfrac{1}{l^2_{j-1}(z,\zeta)d_{j-1}(z,\zeta)d_j(z,\zeta)},\\&
	a_{j-1}(z,\zeta)=\cfrac{1}{d_j(z,\zeta)}\left(z\zeta m_j(z,\zeta)-\cfrac{1}{l_{j-1}(z,\zeta)}-\cfrac{1}{l_j(z,\zeta)}\right).
	\end{split}
\end{equation}

 Moreover,  \eqref{14p.7.3@} can be rewritten as

\begin{equation}\label{14p.7.9@}
 	F(z,\zeta)=\cfrac{Q^+_{2n}(z,\zeta)+Q^+_{2n-1}(z,\zeta)\tau(z,\zeta)}{P^+_{2n}(z,\zeta)+P^+_{2n-1}(z,\zeta)\tau(z,\zeta)},
  \end{equation}
 where the parameter $\tau$ satisfies \eqref{14p.7.4@} and  $P^+_i$ and $Q^+_j$ are defined by \eqref{14p.7.7}--\eqref{14p.7.8}.

\end{theorem}
\begin{proof} The proof of the current statement  follows verbatim from Theorem \ref{14p.th7.2} and Proposition \ref{14p.th7.2s}.~\end{proof}

\section{Full problem}
In the current section, we study the full problem $MP(\mathbf{s}, \mathfrak{N}(\mathbf{s}))$ in the symmetric form. 
 
\textbf{Problem} $MP(\mathbf{s}, \mathfrak{N}(\mathbf{s}))$: Given a sequence of real numbers $\mathbf{s}=\{s_{i,j}\}_{i,j=0}^{\infty}$. Describe the set of functions,  which admit the asymptotic expansion
\begin{equation}\label{14p.5.1gz}
 	F(z,\zeta)=-\sum_{i,j=0}^\infty{i+j\choose i} \frac{s_{i,j}}{z^{i+1}\zeta^{j+1}}=-\sum_{j=0}^{\infty}\cfrac{\mathfrak{s}_j^{(0)}(z,\zeta) }{z^{i+1}\zeta^{j+1}}.
\end{equation}
We study the case, where the associated  sequence $\mathfrak{s}(z,\zeta)=\{\mathfrak{s}_j(z,\zeta)\}_{j=0}^\infty$ is simple regular on the set $ \mathfrak{N}(\mathbf{s})$ and defined by
\begin{equation}\label{14p.5.1g}
\mathfrak s_0^{(0)}(z,\zeta)=s_{0,0}\mbox{ and }\mathfrak s_j^{(0)}(z,\zeta)=\sum\limits_{i=0}^{j-1}{i+j\choose i}\left(s_{i,j}z^{j-i}+s_{j,i}\zeta^{j-i}\right)+{2j\choose j}s_{j,j}.
\end{equation}

\begin{theorem}\label{14p.sym. th5.1} Let $\mathbf{s}=\{s_{i,j}\}_{i,j=0}^{\infty}$ be a sequence of real number and let $\mathfrak{s}(z,\zeta)=\{\mathfrak{s}_j(z,\zeta)\}_{j=0}^\infty$ be an associated sequence of $\mathbf{s}$, which is  a simple regular on the set $ \mathfrak{N}(\mathbf{s})$.  Then any solution of the  problem $MP(\mathbf{s}, \mathfrak{N}(\mathbf{s}))$  admits the following representation
\begin{equation}\label{14p.5.1}
	F(z,\zeta)=-\cfrac{b_0}{a_0(z, \zeta)-\cfrac{b_1(z, \zeta)}{a_1(z, \zeta)-\cdots-\cfrac{b_{n-1}(z, \zeta)}{a_{n-1}(z, \zeta)+\cfrac{b_{n}(z, \zeta)}{\ddots}}}}, 
\end{equation}
where the atoms $(a_j,b_j)$ are defined by~\eqref{14p.3.23}--\eqref{14p.3.25}.
\end{theorem}
\begin{proof}
 Suppose $\mathbf{s}=\{s_{i,j}\}_{i,j=0}^{\infty}$ is the sequence of real number and  $\mathfrak{s}(z,\zeta)=\{\mathfrak{s}_j\}_{j=0}^\infty$ is the associated sequence of $\mathbf{s}$, which is  the simple regular on the set $ \mathfrak{N}(\mathbf{s})$.  By  Proposition~\ref{14p.sym. pro.bw1} and Theorem~\ref{14p.sym. pro.th1}, we obtain
\[
	F(z,\zeta)=-\cfrac{b_0}{a_0(z,\zeta)+F_1(z,\zeta)},
\]
where 
\[
 b_0\!=\!\mathfrak{s}_0^{(0)}(z,\zeta) \!=\!s_{0,0},\, a_{0}(z,\zeta)\!=\!z\zeta-\cfrac{\mathfrak{s}_1^{(0)}(z,\zeta) }{\mathfrak{s}_0^{(0)}(z,\zeta) }\quad\mbox{and}\quad F_{1}(z,\zeta)\!=\!-\!\sum_{j=0}^{\infty}\cfrac{\mathfrak{s}_j^{(1)}(z,\zeta) }{z^j\zeta^j },
\]
where $\mathfrak{s}^{(1)}(z,\zeta)=\{\mathfrak{s}^{(1)}_j(z,\zeta) \}_{j=0}^\infty$  is define by~\eqref{14p.3.25}.

By induction, we get the representation \eqref{14p.5.1} for  $F$. This completes the proof.~
\end{proof}
\begin{remark} According to \cite{Wall}, the continued fraction \eqref{14p.5.1} converges if and only if
\[
	|a_n(z,\zeta)|\geq |b_n(z,\zeta)|+1,\quad \mbox{for all } n\in \mathbb{N}.
\]
\end{remark}



\begin{thebibliography}{[10]}
\bibitem{Akh}
 {N.\,I. Akhiezer},
\emph{The classical moment problem}, Oliver and Boyd, Edinburgh, 1965.

\bibitem{DK15}
 {V. Derkach} and  {I. Kovalyov},
 {On a class of generalized Stieltjes continued fractions, Methods of Funct.
  Anal. and Topology} \textbf{21}(4), 315-335, 2015.

\bibitem{DK17}
V. Derkach I. Kovalyov,
 The Schur algorithm for indefinite Stieltjes  moment problem, Math. Nachr.   \textbf{290} (10),  1637--1662, 2017.


\bibitem{Cuyt} A. Cuyt, Z.G. Golub, P.  Milanfar and B. Verdonk, Multidimensional integral inversion, with applications in shape reconstruction, SIAM J. Sci. Comput. \textbf{27}(3), 1058--1070, 2005.

\bibitem{Hamburger} H. Hamburger, \"Uber eine Erweiterung des Stieltjesschen Momentenproblems, Math. Annalen,  \textbf{81},  235--319, 1920.


\bibitem{K17}I. Kovalyov, A truncated indefinite Stieltjes moment problem, Journal of Mathematical Sciences \textbf{224}, 509--529, 2017.

\bibitem{KuPeRovO16} S. Kunis, T. Peter, T. R\"omer, and U. von der Ohe, A multivariate generalization of Prony's method, Linear Algebra Appl. \textbf{490}, 31--47, 2016.

\bibitem{St89}
 {T.\,J. Stieltjes},
Recherches sur les fractions continues,
 {Ann. Fac. Sci. de Toulouse} \textbf{8}, 1--122, 1894.
 
 \bibitem{Schm}
 K. Schm\"udgen, \emph{The Moment Problem}, Springer, New York, 2017.
 
 \bibitem{Wall} H.S. Wall, \emph{Analytic theory of continued fractions}, Chelsey, NY, 1967.
\end{thebibliography}
\end{document}